\documentclass[12pt, leqno, british]{amsart}
\usepackage[
	style=alphabetic,
	citestyle=alphabetic,
	backend=biber
]{biblatex}

\usepackage{harmony}
\usepackage{a4, amsmath}
\usepackage{mathtools}
\usepackage{amssymb}
\usepackage{epsfig}
\usepackage{graphics}
\usepackage{amsthm, amscd, mathdots}
\usepackage{enumerate}
\usepackage{url}
\usepackage{cancel}
\usepackage{hyperref}
\usepackage{cleveref}
\usepackage{thmtools}
\usepackage{csquotes}
\usepackage{color}
\usepackage{datetime}
\usepackage{todonotes}
\usepackage{stmaryrd}
\usepackage{ifthen}
\usepackage{tikz-cd,tikz}

\usepackage[pass]{geometry}

\usepackage[all]{xy}

\theoremstyle{plain}
\newcounter{thmIcounter}

\newtheorem{thmI}[thmIcounter]{Theorem}
\theoremstyle{definition}
\newtheorem{defi}{Definition}[section]

\theoremstyle{plain}
\newtheorem{prop}[defi]{Proposition}
\newtheorem{lem}[defi]{Lemma}
\newtheorem{thm}[defi]{Theorem}
\newtheorem{cor}[defi]{Corollary}

\newtheorem*{ques*}{Question}
\newtheorem*{thm*}{Theorem}
\theoremstyle{remark}
\newtheorem{opm}[defi]{Remark}

\newcommand{\nat}{\mathbb{N}}
\newcommand{\zz}{\mathbb{Z}}
\newcommand{\rr}{\mathbb{R}}
\newcommand{\ff}{\mathbb{F}}

\newcommand{\qq}{\mathbb{Q}}

\newcommand{\car}{\mathsf{char}}

\newcommand{\Lar}{{\mathcal{L}_{\rm ring}}}

\newcommand{\mc}{\mathcal}
\newcommand{\mf}{\mathfrak}
\newcommand{\mbb}{\mathbb}

\DeclareMathOperator{\Tr}{Tr}

\DeclareMathOperator{\Gal}{Gal}
\DeclareMathOperator{\GL}{GL}
\DeclareMathOperator{\SB}{SB}
\DeclareMathOperator{\Disc}{Disc}

\title[$\exists$-defining valuations in function fields over large fields]{Existentially defining valuations in function fields over large fields}
\author{Nicolas Daans}
\address{KU Leuven, Department of Mathematics, Celestijnenlaan 200B, 3001 Heverlee, Belgium}
\email{nicolas.daans@kuleuven.be}
\address{Université de Mons, Département de Mathematique, Place du Parc 20, 7000 Mons, Belgium}
\email{nicolas.daans@umons.ac.be}
\newdate{v2}{4}{11}{2026}
\date{\displaydate{v2}}

\begin{document}
\begin{abstract}
Let $K$ be a large field such that $K[\sqrt{-1}]$ is not algebraically closed and $F/K$ a function field in one variable.
Extending techniques and results from earlier work with Becher and Dittmann, we show that every valuation ring on $F$ containing $K$ is existentially definable in the language of rings with parameters from $F$.
As a consequence, we obtain the undecidability of the existential theory of $F$ in the language of rings with appropriately chosen parameters.
\end{abstract}
\maketitle

\section{Introduction}

Given a field $F$ and a subring $R$, we call a set $S \subseteq F$ \emph{existentially definable (with parameters in $R$)} if there exist $k, m \in \nat$ and $f_1, \ldots, f_k \in R[X, Y_1, \ldots, Y_m]$ such that
$$ S = \{ x \in F \mid \exists y \in F^m : f_1(x, y) = \ldots = f_k(x,y) = 0 \}.$$
Equivalently, $S$ is existentially definable with parameters in $R$ if and only if it is definable by an existential formula in the first-order language of rings and with parameters from $R$.
When the subring $R$ is not mentioned, we implicitly assume $R = F$.

In this paper, we shall consider the situation where $F$ is a \emph{function field in one variable} over a base field $K$, i.e.~a finitely generated extension of transcendence degree $1$.
Equivalently, $F$ is isomorphic to a finite extension of the \emph{rational function field} $K(T)$, the field of fractions of the polynomial ring in one variable $K[T]$.
We show the following.
\begin{thmI}\label{TI:definability}
Let $K$ be a large field such that $K[\sqrt{-1}]$ is not algebraically closed.
Let $F/K$ be a function field in one variable. Any valuation ring $\mc{O}$ of $F$ with $K \subseteq \mc{O}$ is existentially definable in $F$.
\end{thmI}
Recall that a field $K$ is called \emph{large} (some authors use the term \emph{ample}) if every algebraic curve over $K$ has either zero or infinitely many smooth $K$-rational points.
See \cite{Pop_LittleSurvey} for an overview of equivalent characterisations and examples of large fields.
Examples of large fields include all infinite algebraic extensions of finite fields, arbitrary henselian valued fields, separably closed fields, pseudo-algebraically closed fields, and so-called $p$-closed fields ($p$ prime), i.e.~fields of which all finite extensions have a degree divisible by $p$.
Examples of non-large fields include all number fields, finite fields, and function fields over arbitrary base fields.

\Cref{TI:definability} is obtained as a special case of \Cref{T:Edefinability} below.
\Cref{T:Edefinability} yields more precise control over the parameters involved in the existential definition, and works under slightly more general hypotheses; see also \Cref{R:Uniformity}.

\Cref{TI:definability} fits in with a long history of results on existential definability of valuation rings in function fields.
In \cite{DenefDiophantine}, Denef showed that, for a real field $K$ (like $\qq$ or $\rr$) and $F = K(T)$, the set of functions without a pole at infinity (i.e.~the valuation ring at infinity) is existentially definable with parameters in $\zz[T]$.
Existential definability of valuations in several other families of function fields was considered in the decades thereafter, including certain rational function fields in positive characteristic, in particular over finite base fields \cite{PheidasHilbert10,KimRoush-charp, Videla}, rational function fields over certain subfields of $p$-adic fields \cite{KimRoush_DiophantineUnsolvabilityPAdic,DD12}, as well as general function fields in one variable over real fields or certain subfields of $p$-adic base fields \cite{Mor05,EisHil10padic,MillerShlapentokh_v2}, and in positive characteristic over algebraic extensions of finite fields \cite{ShlapentokhGlobal,Shlapentokh_Pacific00,Eis03,EisShlap17}.

\Cref{TI:definability} generalises the recent \cite[Theorem I]{BDD} (without the uniformity aspects discussed there) by removing the assumption that $K[\sqrt{-1}]$ has a separable extension of even degree.
In particular, \Cref{TI:definability} applies to all base fields $K$ which are not algebraically closed and which are $p$-closed for some prime $p$.
This includes for example separably closed fields which are not algebraically closed - a case which had not been covered before.

While \Cref{TI:definability} allows us to recover several results from the literature (like in the case where $K$ is a $p$-adic field), it does not cover all known cases where $F/K$ is a function field in one variable with an existentially definable valuation ring, like in the case where $K = \qq$ or $K = \ff_p$, since number fields and finite fields do not contain a large subfield.
A unified theory, which yields existential definability of valuations in function fields in one variable over a natural class of base fields $K$ covering the known examples, seems desirable.

Let us now briefly comment on the role of the hypotheses of \Cref{TI:definability}.
For algebraically closed fields $K$, especially if $\car(K) = 0$, the cohomological triviality leaves little expressive room to existentially define interesting sets in function fields in one variable over $K$. Essentially the same difficulty arises when considering function fields over $\rr$ without a real place (like $\rr(T)(\sqrt{-(1 + T^2))}$).
While we do not know whether the statement of \Cref{TI:definability} would be false without the assumption that $K[\sqrt{-1}]$ is not algebraically closed, all known constructions of existential definitions of valuation rings in function fields seem to depend in an essential way on the fact that the residue fields of these valuation rings are not algebraically closed.
In the present paper, this assumption will manifest itself through the existence of a (cyclic) central simple algebra over the function field $F$ which does not split over the completion of the valuation of which we want to existentially define the valuation ring.

It is known that a field $K$ is large if and only if it is existentially closed inside the field of formal Laurent series $K(\!(t)\!)$, which is a complete discretely valued field.
For function fields over complete discretely valued fields, splitting of central simple algebras can be studied via local-global principles based on a technique called field patching, see \Cref{sect:csa-ff}.
We shall thus first establish the validity of the desired formula in the case where $K$ is a complete discretely valued field.
The proof of \Cref{TI:definability} for general large base fields will then be carried out via a reduction to the case of complete discretely valued base fields by using the aforementioned existential closedness property.

We conclude by stating an application to a variation of Hilbert's 10th Problem for function fields.
Given a computable commutative ring $R_0$ and a commutative $R_0$-algebra $R$, let us say that the \emph{existential theory of $R$ with coefficients in $R_0$ is decidable} (respectively \emph{undecidable}) if there exists (respectively does not exist) an algorithm which can take an arbitrary tuple of polynomials $f_1, \ldots, f_k, g_1, \ldots, g_l \in R_0[X_1, \ldots, X_n]$ as input and decide whether there exists $x \in R^n$ such that $f_1(x) = \ldots = f_k(x) = 0$ and $g_1(x), \ldots, g_l(x) \neq 0$.
Hilbert's original problem concerned the case $R = R_0 = \zz$, and it followed from the work of Robinson, Davis, Putnam, and Matiyasevich \cite{Mat70} that the existential theory of $\zz$ with coefficients in $\zz$ is undecidable.
See the introduction to \Cref{sect:undecidability} for more discussion and precise definitions.

We will show the following:
\begin{thmI}\label{TI:decidability}
Consider fields $K_1 \subseteq K$ where $K_1$ is large.
Let $F/K$ be a function field in one variable, and $F_0 \subseteq F$ a finitely generated subfield such that $F$ is generated by $F_0$ and $K$.
Assume that $K[\sqrt{-1}]$ does not contain the algebraic closure of $F_0 \cap K_1$.

Then the existential theory of $F$ with coefficients in $F_0$ is undecidable.
\end{thmI}

\Cref{TI:decidability} follows from \Cref{T:Edefinability} (the more refined version of \Cref{TI:definability}) via a well-known reduction technique.
We will carefully state this general technique in \Cref{P:H10-reduction}, providing also a sketch of the proof in the characteristic-$0$ case.
We also refer to the beginning of \Cref{sect:undecidability} for a discussion of the history of this technique, which has in fact served as a primary motivation for the study of existential definability of valuation rings in function fields in \cite{DenefDiophantine,PheidasHilbert10,Videla} and many of the works mentioned before.

Neither \Cref{TI:definability} nor \Cref{TI:decidability} achieves the goal of providing a unified framework with which we can understand the current state of the art regarding existential definability of valuations rings, or existential decidability in function fields.
For example, we cannot derive from these the existential definability of valuation rings in $F = \qq(T)$, or the undecidability of the existential theory of $F$ - both of which we know to hold.
We have argued before why the case of function fields over algebraically or real closed base fields may be especially subtle, but the role of the largeness assumption is much less convincingly crucial.
While in \Cref{TI:definability} the hypotheses must contain \textit{some} condition beyond the non-algebraic closedness of $K[\sqrt{-1}]$ for the conclusion to hold, it is not at all clear whether this is the case for \Cref{TI:decidability}.
In fact, already in the case of rational function fields, we would expect, but cannot give in general, a positive answer to the following.
\begin{ques*}
Let $K$ be a field, $K_0 \subseteq K$ a finitely generated subfield such that $K$ does not contain the algebraic closure of $K_0$.
Does it follow that the existential theory of $K(T)$ with coefficients in $K_0(T)$ is undecidable?
\end{ques*}
In \Cref{C:decidability-alg-extension} we show that, when $K$ and $K_0$ are as in the setup of the above Question, then there always exists an algebraic extension $K'/K$ such that the existential theory of $K'(T)$ with coefficients in $K_0(T)$ is undecidable.

\subsection*{Acknowledgements}
The research presented in this paper was realised with financial support from \emph{Research Foundation - Flanders (FWO)}, fellowship 1208225N, as well a Postdoctoral Researcher grant from \emph{Fund for Scientific Research - FNRS}.
The author would also like to thank the Hausdorff Research Institute for Mathematics, Bonn, funded by the DFG (under Germany's Excellence Strategy, EXC-2047/1 – 390685813), for its hospitality during the trimester program ``Definability, decidability, and computability'' while this paper was written.
The author further thanks Karim Johannes Becher, Philip Dittmann and Bjorn Poonen for their inspiring and valuable input.

\subsection*{Tool disclosure}
GPT-5.6 Sol was used during the revision of this manuscript to identify typos, minor mistakes, and inconsistencies.
No generative AI was used during the initial research or writing of this manuscript.

\section{Preliminaries on fields and valuations}
We gather some technical results on (valued) fields and their finite extensions.
\begin{prop}\label{P:reduceToCyclic}
Let $K$ be a field.
\begin{enumerate}
\item If $K[\sqrt{-1}]$ is algebraically closed, then $K$ is either algebraically closed or real closed.
\item If $K[\sqrt{-1}]$ is not algebraically closed, then no finite field extension of $K$ is algebraically closed.
\item If $K[\sqrt{-1}]$ is not algebraically closed, then there exists a separable finite field extension $L/K$ and a prime number $p$ such that $L$ has a normal field extension of degree $p$.
\item If $K[\sqrt{-1}]$ has a normal field extension of prime degree $p$, then so does every finite field extension of $K$.
\end{enumerate}
\end{prop}
\begin{proof}
The first two statements are a well-known result due to Artin and Schreier \cite{AS-real-closed}, see for example \cite[Corollary VI.9.3 and Proposition XI.2.4]{LangAlgebra} for a modern proof.

For the third assertion, first observe that if $K$ is of characteristic $p$ and $\alpha$ is not a $p$-th power, then $K(\alpha^{1/p})/K$ is a normal extension of degree $p$.
It remains to consider the case where $K$ is perfect.
Consider a finite proper Galois extension $M/K$ and let $p$ be a prime dividing the degree $[M : K]$.
Let $L$ be the fixed field of a $p$-Sylow subgroup of the Galois group of $M/K$.
Then $M/L$ is a cyclic extension of degree a power of $p$, hence it contains a cyclic (in particular normal) subextension of degree $p$, as desired.

The fourth statement is clear in the case $K$ is imperfect of characteristic $p$ (since then the same holds for every finite field extension).
If $K[\sqrt{-1}]$ has a cyclic extension of degree $p$, the desired statement follows immediately from \cite[Theorem 2]{WhaplesAlgebraic}.
\end{proof}
For a valuation $v$ on a field $K$ we denote by $\mc{O}_v$, $\mf{m}_v$, $Kv$, $vK$ and $K_w$ the associated valuation ring, valuation ideal, residue field, value group, and henselisation, respectively.
When $K_0$ is a subfield of $K$, we may write $K_0v$, $vK_0$ and $(K_0)_v$ for respectively the residue field, value group, and henselisation of the restriction of $v$ to $K_0$.
In the next section we will sometimes consider \emph{rank one valuations}; these are non-trivial valuations for which the value group can be embedded (as an ordered abelian group) into $\rr$, or equivalently, the value group has no proper non-trivial convex subgroup.
\begin{lem}\label{L:rightpoly}
Let $(K, v)$ be a valued field, and let $p$ be a prime number.
If $\car(Kv) \neq p$, assume that $K$ contains a $p$th root of unity.
If $\car(Kv) = p$, assume that $\car(K) = p$.
Assume that $Kv$ has a normal field extension $\ell$ of degree $p$ over $Kv$.
Fix $m \in \mf{m}_v$.
There exists a monic degree $p$ polynomial $f(X) \in \mc{O}_v[X]$ satisfying the following:
\begin{enumerate}
\item\label{it:cycl1} The splitting field $K'$ of $f$ over $K$ is a cyclic extension of degree $p$.
\item\label{it:cycl2} The reduction of $f$ modulo $\mf{m}_v$ is irreducible with splitting field $\ell$.
\end{enumerate}
\end{lem}
\begin{proof}
First consider the case where $\car Kv \neq p$. Then $\ell$ is a cyclic extension of degree $p$, so by \autocite[Theorem VI.6.2(i)]{LangAlgebra} there exists $a \in \mc{O}_v$ such that $\ell$ is the splitting field of the reduction of $X^p - a$ modulo $\mf{m}_v$.
By \autocite[Theorem VI.6.2(ii)]{LangAlgebra} the splitting field of $X^p - a$ over $K$ is a cyclic extension of degree $p$.
We may hence set $f(X) = X^p - a$.

Now consider the case where $\car Kv = \car K = p$.
We distinguish between the cases where $\ell$ is a separable (hence cyclic) or inseparable degree $p$ extension of $Kv$.
If $\ell /Kv$ is cyclic, then by \autocite[Theorem VI.6.4(i)]{LangAlgebra} there exists $a \in \mc{O}_v$ such that $\ell$ is the splitting field of the reduction of $X^p - X - a$ modulo $\mf{m}_v$.
By \autocite[Theorem VI.6.4(ii)]{LangAlgebra} the splitting field of $X^p - X - a$ over $K$ is a cyclic extension of degree $p$.
Hence we may set $f(X) = X^p - X - a$.
On the other hand, if $\ell/Kv$ is an inseparable extension of degree $p$, then there exists $a \in \mc{O}_v$ such that its reduction $\overline{a}$ modulo $\mf{m}_v$ satisfies $\ell = Kv[\sqrt[p]{\overline{a}}]$.
We may then set $f(X) = X^p  - m^{p-1} X - a$; after applying the base change $X \mapsfrom mX$ one obtains the polynomial $X^p - X - am^{-p}$ and concludes again by \autocite[Theorem VI.6.4(ii)]{LangAlgebra} that the splitting field of $f$ over $K$ is a cyclic extension of degree $p$, as desired.
\end{proof}
For a univariate polynomial $f$, denote by $\Disc(f)$ its polynomial discriminant, see for example \cite[Section IV.8]{LangAlgebra} for a definition and basic properties.
The following technical lemma may be seen as a specialization of Krasner's Lemma from valuation theory.
\begin{lem}\label{L:Krasner}
Let $(K, v)$ be a henselian valued field, $f \in \mc{O}_v[X]$ monic of prime degree $p$ such that the splitting field of $f$ is a cyclic extension of degree $p$.
For $t \in K$ with $p(p-1)v(t) > pv(\Disc(f))-(p-1)^2v(f(0))$, the splitting field of the polynomial $f_t(X) = f(X) + tX^{p-1}$ contains the splitting field of $f$.
\end{lem}
\begin{proof}
Fix $t \in K$ with $p(p-1)v(t) > pv(\Disc(f))-(p-1)^2v(f(0))$.
Let $x_1, \ldots, x_p$ denote the roots of $f$ in its splitting field $K'$, and note that these are pairwise distinct as $f$ is separable.
Let $K''$ be the splitting field of $f_t$ over $K'$. We denote by $v$ also the unique extension of $v$ to $K''$.
Since $f$ has a cyclic splitting field, we may assume that $x_2 = \sigma(x_1)$, $x_3 = \sigma(x_2)$, \ldots, $x_{p} = \sigma(x_{p-1})$ and $x_1 = \sigma(x_p)$ for some $K$-automorphism $\sigma$ of $K'$.
In particular, it follows that $v(x_i) = \frac{1}{p}v(f(0))$ and $v(x_1 - x_2) = v(x_i - x_j)$ for all $i, j \in \lbrace 1, \ldots, p \rbrace$ with $i \neq j$.
From this it follows that $v(\Disc(f)) = v(\prod_{1 \leq i < j \leq p} (x_i - x_j)^2) = p(p-1)v(x_1-x_2)$.

Now let $y_1, \ldots, y_p$ denote the roots of $f_t$ in $K''$.
We compute that
\begin{align*}
p \max_{1 \leq i \leq p} v(x_1 - y_i) &\geq \sum_{i=1}^p v(x_1 - y_i) = v(\prod_{i=1}^p (x_1 - y_i)) = v(f_t(x_1)) \\
&= v(f_t(x_1) - f(x_1)) = v(tx_1^{p-1}) = v(t) + \frac{p-1}{p}v(f(0)) \\
&> \frac{v(\Disc(f))}{p-1} = pv(x_1-x_2) = p \max_{2 \leq i \leq p} v(x_1 - x_i).
\end{align*}
By Krasner's Lemma \cite[Theorem 4.1.7]{Eng05} there exists $i \in \lbrace 1, \ldots, p \rbrace$ such that $K(y_i)$ contains $x_1$, and thus, since $K'/K$ is cyclic, contains $K'$.
\end{proof}
\begin{opm}
If in the setup of \Cref{L:Krasner} one has that $f$ is a separable irreducible polynomial modulo $\mf{m}_v$, then $\Disc(f)$ and $f(0)$ are units, and thus the condition on $t$ reduces to simply $v(t) > 0$.
\end{opm}

We include the following easy lemma on extensions of henselian valued fields for lack of a suitable reference.
For a finite field extension $L/K$, we denote by $\Tr_{L/K}$ and $N_{L/K}$ respectively the \emph{trace} and \emph{norm} maps $L \to K$.
\begin{lem}\label{L:norms-are-open}
Let $(L,w)/(K, v)$ be a finite separable extension of henselian valued fields.
The set $N_{L/K}(L^\times)$ is open with respect to the $v$-adic topology.
\end{lem}
\begin{proof}
Since $N_{L/K}$ is multiplicative, it suffices to show that $N_{L/K}(L^\times)$ contains an open neighbourhood of $1$.
Since $L/K$ is separable, $\Tr_{L/K}$ is not identically zero \cite[Theorem VI.5.2]{LangAlgebra}, so we may fix $\alpha \in \mc O_w$ such that $\Tr_{L/K}(\alpha) \neq 0$.
We claim that for all $x \in K$ with $v(x) > 2v(\Tr_{L/K}(\alpha))$ we have $1+x \in N_{L/K}(L^\times)$.
It remains to prove this claim.

For $\lambda \in K$ arbitrary, one computes (using definitions of norms and trace, see e.g.~\cite[Section VI.5]{LangAlgebra}) that
$$ N_{L/K}(1 + \alpha \lambda) = 1 + \lambda \Tr_{L/K}(\alpha) + \lambda^2 g(\alpha) $$
for some polynomial $g(T) \in \mc{O}_v[T]$.
If $x \in K$ is such that $v(x) > 2v(\Tr_{L/K}(\alpha))$, then by Hensel's Lemma the polynomial $T \Tr_{L/K}(\alpha) + T^2g(T) - x \in \mc{O}_v[T]$ has a simple root $\lambda$ in $K$, and then $1 + x = N(1 + \alpha \lambda)$ as desired.
\end{proof}

\section{Central simple algebras over function fields}\label{sect:csa-ff}

We shall now consider central simple algebras over fields.
Given a field $K$, a \emph{central simple algebra over $K$} is a finite-dimensional simple $K$-algebra $A$ with centre $K$.
The $K$-dimension of such an algebra is always a square; its square root is called the \emph{degree} of the algebra \cite[Section 2.2]{GilleSzamuely_2nd}.
When $L/K$ is an arbitrary field extension and $A$ a central simple $K$-algebra of degree $n$, then $A \otimes_L K$ is naturally a central simple $L$-algebra of degree $n$, which we denote by $A_L$.
Given a field extension $L/K$, we shall say that $A_L$ is \emph{split} (or that $A$ \emph{splits over} $L$) if $A_L$ is $L$-isomorphic to the ring $\mathbb M_n(L)$ of $(n \times n)$-matrices over $L$.
Otherwise, we say that $A_L$ is \emph{non-split} (or that $A$ is \emph{non-split over $L$}).
We refer the interested reader to \cite[Chapter 2]{GilleSzamuely_2nd} for a gentle introduction to the theory of central simple algebras over fields.

Given a central simple algebra $A$ of degree $n$ over $K$, its set of $n$-dimensional right ideals forms a smooth variety called the \emph{Severi-Brauer variety of $A$} \cite[Definition 1.16]{BOI}, which we denote by $\SB(A)$.
It has the property that, for any field extension $L/K$, $A_L$ is split if and only if $\SB(A)$ has an $L$-rational point \cite[Proposition 1.17]{BOI}.
The algebraic group of invertible elements of $A$, denoted $\GL_1(A)$, naturally acts on $\SB(A)$ via left multiplication.

\begin{thm}[Mehmeti]\label{T:LGP-Mehmeti}
Let $(K, v_K)$ be a complete rank one valued field, $F/K$ a function field in one variable, $A$ a central simple algebra over $F$.
Suppose that $A$ is split over $F_w$ for all valuations $w$ on $F$ for which $w\vert_K = v_K$. Then $A$ is split.
\end{thm}
\begin{proof}
We first observe the following: if $w_0$ is a non-trivial valuation on $F$ which is trivial on $K$, then the residue field $Fw_0$ of $w_0$ is isomorphic (over $K$) to a finite extension of $K$, and is thus itself henselian with respect to the unique extension of $v_K$.
Since the composition of two henselian valuations is again henselian \cite[Corollary 4.1.4]{Eng05}, it follows that a henselisation of $w_0$ is also a henselisation of the valuation $w$ on $F$ obtained by composing $w_0$ with the unique extension of $v_K$ to $Fw_0$.
This valuation $w$ satisfies $w\vert_K = v_K$.
The hypotheses thus imply that $A$ is split over $F_{w_0}$ for all non-trivial valuations $w_0$ on $F$ which are trivial on $K$.

As explained in the beginning of section 5 of \autocite{HHK_ApplicationsPatchingToQuadrFormsAndCSAs}, the $F$-linear algebraic group $\GL_1(A)$ is a rational connected linear algebraic group which acts strongly transitively on the points of the Severi-Brauer variety $\SB(A)$ associated to $A$.

By Mehmeti's local-global-principle \autocite[Corollary 3.18(2) and Remark 3.20]{Mehmeti_PatchingBerkQuad} we have that the variety $\SB(A)$ has an $F$-rational point if and only if it has an $F_w$-rational point for each rank one valuation $w$ on $F$ for which either $w\vert_K = 0$ or $w\vert_K = v_K$.
Since the Severi-Brauer variety associated to a central simple algebra has a rational point if and only if the algebra is split, this (together with the observation from the first paragraph) concludes the proof.
\end{proof}

\begin{opm}
If $F$ is a global field (i.e.~a finite extension of $\qq$ or of $\ff_p(T)$ for some prime number $p$), the Albert-Brauer-Hasse-Noether Theorem states that a central simple $F$-algebra $A$ is split if and only if it is split over $F_w$ for all $\zz$-valuations $w$ on $F$ \cite[Theorem 8.1.17]{Neu15}.
In other words, $\SB(A)$ has an $F$-rational point if and only if it has an $F_w$-rational point for every $\zz$-valuation $w$ on $F$.
The Albert-Brauer-Hasse-Noether Theorem, and the closely related Hasse Norm Theorem \cite[Corollary VI.4.5]{Neu99}, are examples of \emph{local-global principles}.
They have been used extensively in the study of (existentially) definable subsets of global fields, see for example \cite{Rumely,Shlapentokh_Pacific00,EisShlap17} for applications of the Hasse Norm Theorem to obtain a.o.~definability of valuation rings in global fields, \cite{EisentragerThesis,Poo09,Par13,Koe16,Eis18,DaansGlobal} for the usage of quaternion algebras (i.e.~central simple algebras of degree $2$), and \cite{Dit17} for the usage of more general central simple algebras to obtain several natural existentially definable predicates in global fields.

\Cref{T:LGP-Mehmeti} gives a local-global principle for splitting of central simple algebras over function fields over complete valued fields. It is a specialisation of a much more general class of local-global principles using a technique called \emph{field patching}, going back to the foundational work of Harbater, Hartmann and Krashen \cite{HHK_ApplicationsPatchingToQuadrFormsAndCSAs}, and adapted to the setup of Berkovich geometry by Mehmeti \cite{Mehmeti_PatchingBerkQuad}. 
Local-global principles based on patching have been applied to study a.o.~properties of quadratic forms and central simple algebras over function fields over complete valued fields, see e.g.~\cite{CTPS12,BGvG14,Parimala_Preeti_Suresh_2018,DaansLinkage}.
In \cite{BDD} a related local-global principle for isotropy of quadratic forms was used to study existential definability of certain subrings of function fields over complete discretely valued fields.
Several of the arguments presented in this manuscript are a variation thereof.
\end{opm}

For a cyclic field extension $L/K$ of degree $m \in \nat$, a generator $\sigma$ of $\Gal(L/K)$ and an element $\pi \in F^\times$, the algebra $A = (L/K, \sigma, \pi)$ is defined to be the $m^2$-dimensional $K$-algebra generated by $L$ and an element $\alpha$ subject to the relations $x \alpha = \alpha \sigma(x)$ for $x \in L$ and $\alpha^m = \pi$.
It is a central simple $K$-algebra of degree $m$, and it is split if and only if $\pi$ is a norm of $L/K$ \cite[15.1, Lemma]{Pie82}.
Such an algebra is called a \emph{cyclic algebra}.

Consider a function field in one variable $F/K$.
We denote by $\mc{V}(F/K)$ the set of $\zz$-valuations on $F$ which are trivial on $K$.
If $A$ is a central simple $F$-algebra, we denote by $\Delta_K A$ the set of $w \in \mc{V}(F/K)$ for which $A$ is non-split over $F_w$.

When two valuations $v_1$ and $v_2$ on a field $F$ have the same associated valuation ring (i.e.~$\mc{O}_{v_1} = \mc{O}_{v_2}$), then we call $v_1$ and $v_2$ \emph{equivalent}, which we may denote by $v_1 \sim v_2$.
For a field $F$, let $\mathbb{V}(F)$ be the set of equivalence classes of valuations on $F$.
For a subset $S$ of $\mathbb{V}(F)$, we will abuse notation and write ``$v \in S$'' to mean ``$v$ is a valuation on $F$ whose equivalence class lies in $S$''.
We consider $\mathbb{V}(F)$ as a topological space with respect to the constructible topology.
By definition, the constructible topology is the coarsest topology on $\mathbb{V}(F)$ such that, for each $a \in F$, the set $ \{ v \in \mathbb{V}(F) \mid a \in \mf m_v \}$ is both open and closed.
The obtained topological space is a compact Hausdorff space, see e.g.~\cite[Section 1]{HuberKnebusch}.

The following two lemmata are variations of \cite[Lemma 8.1 and Lemma 8.3]{BDD} and are proven by completely identical arguments (with ``anisotropic quadratic form'' replaced by ``non-split central simple algebra'' each time).
\begin{lem}\label{L:Delta-compact}
Let $F$ be a field, $A$ a central simple algebra over $F$.
Then $\lbrace v \in \mathbb{V}(F) \mid A_{F_v} \text{ is non-split} \rbrace$ is a closed and compact subspace of $\mathbb{V}(F)$.
\end{lem}
\begin{lem}\label{L:coarsen-rank-1}
Let $F/K$ be a function field in one variable, $v \in \mathbb{V}(F)$ such that $v\vert_K$ is a rank one valuation, and $d \in K^\times \cap \mf m_v$.
Let $A$ be a central simple algebra over $F$ such that $A_{F_v}$ is non-split.
Then either $\mc{O}_v[d^{-1}] = F$ or $\mc{O}_v[d^{-1}] = \mc{O}_w$ for some $w \in \Delta_K A$.
\end{lem}

\section{Existential definability in function fields}\label{sect:definability}
We will find it convenient to phrase our results in the set-up of first-order logic.
We shall work with the signature of rings $\Lar$, consisting of three binary function symbols $+, -, \cdot$ as well as two constant symbols $0, 1$. We will consider fields as $\Lar$-structures by interpreting these symbols in the obvious way.
When $F$ is a field and $R_0$ a subset of $F$, we will also consider $\Lar(R_0)$, which is the signature obtained by adjoining to $\Lar$ a constant symbol $c_r$ for each $r \in R_0$.
We can then naturally consider $F$ as an $\Lar(R_0)$-structure by interpreting $c_r$ as $r$.

Given a field $F$ and $n \in \nat$, a subset of $F^n$ will be called \emph{existentially $\Lar$-definable} (or $\exists$-$\Lar$-definable for short) if it is definable by an $\Lar$-formula of the form $\exists y_1, \ldots, y_m \psi$ where $\psi$ is a quantifier-free formula.
Similarly, we define existentially $\Lar(R_0)$-definable ($\exists$-$\Lar(R_0)$-definable) subsets of $F^n$ for a subset $R_0 \subseteq F$.
This definition is compatible with the one alluded to in the Introduction: indeed, if $R_0$ is a subring of $F$, then a subset $S$ of $F^n$ is $\exists$-$\Lar$-definable (respectively $\exists$-$\Lar(R_0)$-definable) if and only if there exist $k, m \in \nat$ and $f_1, \ldots, f_k \in \zz[X_1, \ldots, X_n, Y_1, \ldots, Y_m]$ (respectively $\in R_0[X_1, \ldots, X_n, Y_1, \ldots, Y_m]$) such that
$$ S = \{ x \in F^n \mid \exists y \in F^m : f_1(x, y) = \ldots = f_k(x,y) = 0 \}.$$
See for example \cite[Remark 3.4]{DDF} for a discussion of the (essentially well-known) equivalence of these two definitions.
Clearly, a finite intersection or finite union of $\exists$-$\Lar$-definable ($\exists$-$\Lar(R_0)$-definable) subsets of $F^n$ is again an $\exists$-$\Lar$-definable ($\exists$-$\Lar(R_0)$-definable) subset of $F^n$.

When $\varphi$ is an $\Lar$-formula ($\Lar(R_0)$-formula) and $x_1, \ldots, x_n$ are distinct variable symbols such that no variable other than $x_1, \ldots, x_n$ appears freely in $\varphi$, we might introduce the formula $\varphi$ as $\varphi(x_1, \ldots, x_n)$ to convey this information in a compact way.
When $F$ is an $\Lar$-structure ($\Lar(R_0)$-structure) we may then write $\varphi(F)$ to mean the subset $ \lbrace (a_1, \ldots, a_n) \in F^n \mid F \models \varphi (a_1, \ldots, a_n) \rbrace$; if $F$ is a field and $\varphi$ is an existential $\Lar$-formula (existential $\Lar(R_0)$-formula), then $\varphi(F)$ is by definition $\exists$-$\Lar$-definable ($\exists$-$\Lar(R_0)$-definable).

\begin{prop}\label{P:Koenigsmann-defining-constants}
Let $\mf g \in \nat$.
There is an existential $\Lar$-formula $\varphi_{\mf g}(x)$ such that, for every field $K$ and every function field in one variable $F/K$ of genus at most $\mf g$, one has $\varphi_{\mf g}(F) \subseteq K$, and furthermore equality holds when $K$ is large.
\end{prop}
\begin{proof}
By \cite[Theorem 2]{Koe02} there exists for every function field in one variable $F$ over a large field $K$ an existential $\Lar$-formula $\varphi_{F}(x)$ such that $\varphi_{F}(F) = K$.
Inspection of the proof readily reveals that the largeness assumption is only needed for the inclusion $\varphi_F(F) \supseteq K$, and that such a formula $\varphi_F$ can be chosen to depend on $F$ only via the equation of a plane curve with integral coefficients, with a rational point in the prime field of $F$, and with genus strictly higher than that of $F$.
Letting ${\mf g}$ be an upper bound for the genus of $F$, as explained in \cite[Example 9.5]{BDD}, one may take the plane curve with equation
\begin{itemize}
\item $Y^2 - Y = X^{2 {\mf g} + 3}$ if $\car(F) = 2$,
\item $Y^2 = X^{2{\mf g}+3} - X + 1$ if $\car(F)$ divides $2{\mf g} + 3$,
\item $Y^2 = -X^{2{\mf g}+3} + 1$ if $\car(F)$ does not divide $2(2{\mf g}+3)$.
\end{itemize}
By making a finite case distinction on the characteristic of $F$, we may thus find a single existential $\Lar$-formula $\varphi_{\mf g}$ which has the desired property for all function fields in one variable $F/K$ of genus at most ${\mf g}$.
\end{proof}
\begin{prop}\label{P:mainProp}
Let $F/K$ be a function field in one variable, $f \in F[X]$ monic of prime degree $p$ such that its splitting field $F'$ is a cyclic extension of degree $p$ of $F$.
Let $A$ be a central simple $F$-algebra of degree $p$ such that $A_{F'}$ is split and, for all $w \in \Delta_K A$, we have $f \in \mc{O}_w[X]$. 
For $t \in F$, define $f_t(X) = f(X) + tX^{p-1} \in F[X]$.
Finally, let $M \in \nat$ be such that $M > \max_{w \in \Delta_K A} w(\Disc(f))$.
For $t \in F$ we have
$$ t \in \bigcap_{w \in \Delta_K A} \mc{O}_w \enspace\Leftarrow\enspace \exists a \in K^\times : A \text{ is split over the splitting field of } f_{at^M\Disc(f)}, $$
and the other implication holds when $K$ is large.
\end{prop}
\begin{proof}
Consider $t \in F$.

Assume first that $a \in K^\times$ is such that $A$ is split over the splitting field of $f_{at^M\Disc(f)}$.
Consider an arbitrary valuation $w \in \Delta_K A$; note that $w(at) = w(t)$.
Suppose for the sake of a contradiction that $w(t) < 0$.
Then also $w(at^M\Disc(f)) = M w(t) + w(\Disc(f)) < 0$ in view of the choice of $M$.
Since $f \in \mc{O}_w[X]$, it follows that $f_{at^M\Disc(f)}$ has a root over $F_w$ by henselianity.
But by assumption $A_{F_w}$ is non-split, whereby it cannot become split over the splitting field of $f_{at^M\Disc(f)}$ (since this splitting field has a degree over $F_w$ which is coprime to $p$, the degree of $A_{F_w}$).
This contradicts the choice of $a$. We infer that $w(t) \geq 0$.
This concludes the proof of the implication from right to left.

For the other implication, first consider the special case where $K$ is complete with respect to a $\zz$-valuation $v_K$.
We consider the set
$$ \Delta = \lbrace v \in \mathbb{V}(F) \mid v\vert_K \sim v_K \text{ and } A_{F_v} \text{ is non-split}\rbrace.$$
This is a closed and hence compact subspace of $\mathbb{V}(F)$:
the closedness of the condition that $A_{F_v}$ is non-split is given by \Cref{L:Delta-compact}, and the condition $v\vert_K \sim v_K$ corresponds to the closed set
$$ \bigcap_{b \in \mc{O}_{v_K}} \{ v \in \mathbb{V}(F) \mid b \in \mc{O}_v \} \cap \{ v \in \mathbb{V}(F) \mid \tau \in \mf{m}_v \} $$
where $\tau$ is any fixed non-zero element of $\mf{m}_{v_K}$.
By \Cref{L:coarsen-rank-1} we have that, for each $v \in \Delta$, either $\mc{O}_v[\tau^{-1}] = F$ or $\mc{O}_v[\tau^{-1}] = \mc{O}_w$ for some $w \in \Delta_K A$.
By \cite[Lemma 8.2]{BDD} we thus have
$$ \bigcap_{w \in \Delta_K A} \mc{O}_w = \bigcap_{v \in \Delta} \mc{O}_v[\tau^{-1}] = \bigcup_{n \in \nat} \tau^{-n} \bigcap_{v \in \Delta} \mf{m}_v.$$
We have that $f(0), \Disc(f) \in \bigcap_{w \in \Delta_K A} \mc{O}_w$ by assumption, and thus the hypothesis $t \in \bigcap_{w \in \Delta_K A} \mc{O}_w$ implies that also $\frac{(t^M \Disc(f))^{p(p-1)}f(0)^{(p-1)^2}}{\Disc(f)^p} \in \bigcap_{w \in \Delta_K A} \mc{O}_w$, which by the above yields the existence of $n \in \nat$ for which $\frac{(\tau^n t^M \Disc(f))^{p(p-1)}f(0)^{(p-1)^2}}{\Disc(f)^p} \in \bigcap_{v \in \Delta} \mf{m}_v$.
In other words, for all $v \in \Delta$ we have $p(p-1)v(\tau^{n}t^M \Disc(f)) > pv(\Disc(f)) - (p-1)^2v(f(0))$, which in view of \Cref{L:Krasner} implies that $A$ is split over the splitting field of $f_{\tau^{n}t^M \Disc(f)}$ over $F_w$.
Now use \Cref{T:LGP-Mehmeti} to conclude in the special case where $K$ is complete with respect to a $\zz$-valuation $v_K$.

Now assume just that $K$ is large.
Consider $\tilde{K} = K(\!(t)\!)$, the field of formal Laurent series over $K$.
Then $\tilde K$ is complete with respect to a $\zz$-valuation, and $K$ is existentially closed in $\tilde{K}$ (denoted $K \prec_\exists \tilde K$) \cite[Proposition 2.4]{Pop_LittleSurvey}, i.e.~every existential $\Lar(K)$-sentence which holds in $\tilde K$, also holds in $K$.
In particular $\tilde K/K$ is regular, and we may consider the compositum $\tilde F = F\tilde{K}$, which is a function field in one variable over $\tilde K$.
We have that $F \prec_\exists \tilde F$ \cite[Lemma 7.2]{DDF}.
Also, by \Cref{P:Koenigsmann-defining-constants} there is an existential $\Lar$-formula which defines $K$ in $F$ and $\tilde K$ in $\tilde F$.
We infer the following sequence of implications:
\begin{align*}
&t \in \bigcap_{w \in \Delta_K A} \mc{O}_w, \\
\enspace\Rightarrow\enspace &t \in \bigcap_{w \in \Delta_{\tilde K} A_{\tilde F}} \mc{O}_w, \\
\enspace\Rightarrow\enspace &\exists a \in \tilde K^{\times} : A_{\tilde F} \text{ is split over the splitting field of } f_{at^M\Disc(f)} \text{ over } \tilde F, \\
\enspace\Rightarrow\enspace &\exists a \in K^\times : A \text{ is split over the splitting field of } f_{at^M\Disc(f)} \text{ over } F.
\end{align*}
Here, the first implication follows because the henselisation of a valuation in $\Delta_{\tilde K} A_{\tilde F}$ definitely contains the henselisation of its restriction to $F$, and thus $\Delta_K A \supseteq \lbrace w\vert_{F} \mid w \in \Delta_{\tilde K} A_{\tilde F} \rbrace \setminus \lbrace 0 \rbrace$ (where $0$ denotes the trivial valuation on $F$). The second implication follows from the already established case above, since $\tilde K$ carries a complete $\zz$-valuation. The last implication follows from the fact that $F \prec_\exists \tilde F$ and the existential formula defining $\tilde K$ in $\tilde F$ also defines $K$ in $F$.
\end{proof}
\begin{thm}\label{T:Edefinability}
Consider fields $K_1 \subseteq K$ where $K_1$ is large.
Let $F/K$ be a function field in one variable, and $F_0 \subseteq F$ a subfield such that $K[\sqrt{-1}]$ does not contain the algebraic closure of $K_0 = F_0 \cap K_1$.
Let $v \in \mc V(F/K)$ and assume that $F_0/K_0$ is a function field in one variable containing a uniformiser for $v$.

There exists $D \subseteq F$ existentially definable in $F$ with parameters in $F_0$ and such that
$$ \bigcap_{w \in C} \mc{O}_w \cap F_1 \subseteq D \subseteq \bigcap_{w \in C} \mc{O}_w$$
where $C$ is the set of $w \in \mc{V}(F/K)$ with $w\vert_{F_0} \sim v\vert_{F_0}$ and $F_1$ is any subfield of $F$ containing $F_0K_1$ such that $F_1/K_1$ is a function field in one variable.
\end{thm}
\begin{proof}
The relative algebraic closure of $K$ inside $F$ is a finite field extension of $K$.
We may thus without loss of generality replace $K$ by its relative algebraic closure in $F$ and $K_1$ by its relative algebraic closure in $K$. Indeed, this does not affect the hypothesis that $K_1$ is large (since an algebraic extension of a large field is large \cite[Proposition 2.7]{Pop_LittleSurvey}).
In view of \Cref{P:reduceToCyclic}, it furthermore does not affect the hypothesis that $K[\sqrt{-1}]$ does not contain the algebraic closure of $F_0 \cap K_1$.
After this reduction, we obtain that $K_0 = F_0 \cap K_1$ is relatively algebraically closed inside $F_0$.

\textbf{Special case:} We first consider the special case where each of the following assumptions is satisfied:
\begin{itemize}
\item $Fv$ has a normal field extension of degree $p$ for some prime number $p$, and this field extension can be given as the root field of a degree $p$ polynomial defined over $F_0'v$, where $F_0' = F_0 K_0'$ for some finite Galois extension $K_0'/K_0$ contained in $K$ (hence in $K_1$),
\item if $\car(K) \neq p$, then $K_0$ contains a $p$th root of unity.
\end{itemize}

\noindent
\begin{minipage}[l]{0.5\textwidth}
We now apply \Cref{L:rightpoly} to $F_0'$ and the valuation $v\vert_{F_0'}$ to obtain $f \in (\mc{O}_{v} \cap F_0')[X]$ such that the splitting field $F_0'^{[1]}$ of $f$ is a cyclic extension of degree $p$ of $F_0'$, and $f$ is irreducible modulo $\mf{m}_{v}$.
\end{minipage}
\begin{minipage}[r]{0.5\textwidth}
\[\begin{tikzcd}[cramped]
	{F_0'^{[1]}} & \cdots & {F_0'^{[m]}} & \\
	{F_0} & {F_0'} & {F_1} & F \\
	{K_0} & {K_0'} & {K_1} & K
	\arrow[from=2-1, to=2-2]
	\arrow[from=2-2, to=1-1]
	\arrow[from=2-2, to=1-2]
	\arrow[from=2-2, to=1-3]
	\arrow[from=2-2, to=2-3]
	\arrow[from=2-3, to=2-4]
	\arrow[from=3-1, to=2-1]
	\arrow[from=3-1, to=3-2]
	\arrow[from=3-2, to=2-2]
	\arrow[from=3-2, to=3-3]
	\arrow[from=3-3, to=2-3]
	\arrow[from=3-3, to=3-4]
	\arrow[from=3-4, to=2-4]
\end{tikzcd}\]
\end{minipage}
Since $K_0'/K_0$ is Galois, it is generated by some element $\alpha_1$.
Let $\alpha_1, \ldots, \alpha_m$ denote all roots of $\alpha_1$'s minimal polynomial $h(X) \in K_0[X]$.
We may further find $g(X, Y) \in (\mc{O}_v \cap F_0)[X, Y]$ such that $f(X) = g(X, \alpha_1)$.
Observe that each $\alpha_i$ generates $F_0'$ over $F_0$.
For $i \in \{ 1, \ldots, m \}$, denote by $F_0'^{[i]}$ the splitting field of $g(X, \alpha_i)$, and note that all of these are cyclic degree $p$ extensions of $F_0'$.

Let $S \subset \mc{V}(F_0/K_0)$ be the set of restrictions to $F_0$ of valuations $w' \in \mc{V}(F_0'/K_0')$ for which there exists $i \in \lbrace 1, \ldots, m \rbrace$ such that either one of the coefficients of $g(X, \alpha_i)$ does not lie in $\mc{O}_{w'}$, or the discriminant of $g(X, \alpha_i)$ lies in $\mf{m}_w$.
Note that $S$ is finite (since for every non-zero $a \in F_0'$, there are only finitely many $w' \in \mc{V}(F_0'/K_0')$ for which $a \not\in \mc{O}_{w'}^\times$).
By Strong Approximation \cite[Theorem 1.6.5]{Sti09} we may find $\pi \in F_0^\times$ such that, for all $i \in \lbrace 1, \ldots, m \rbrace$,
\begin{itemize}
\item $v(\pi) = 1$,
\item for any $w \in S \setminus \lbrace v \rbrace$, $\pi$ is a norm of ${F'_0}^{[i]}(F_0')_{w'}/(F_0')_{w'}$ for all extensions $w'$ of $w$ to $F_0'$ (by \Cref{L:norms-are-open} it suffices to let $w(\pi - 1)$ be high enough),
\item for any $w \in \mc{V}(F_0/K_0) \setminus \{ v \}$, if $g(X, \alpha_i)$ is irreducible over $(F_0')_{w'}$ for some extension $w'$ of $w$, then $w(\pi) \leq 0$.
\end{itemize}

For each $i \in \{ 1, \ldots, m \}$, fix a generator $\sigma^{[i]}$ of $\Gal({F'_0}^{[i]}/F_0')$ and define the cyclic $F'_0$-algebra $A^{[i]} = ({F'_0}^{[i]}/F'_0, \sigma^{[i]}, \pi)$.
We have that $v \in \Delta_K A_F^{[i]}$: indeed we have $A_{F_v}^{[i]} \cong ({F_0'}^{[i]}F_v/F_v, \sigma^{[i]}, \pi)$, and since $g(X, \alpha_i)$ is irreducible modulo $\mf{m}_v$, the norm map $N_{{F'_0}^{[i]}F_v/F_v}$ represents only elements whose $v$-value is a multiple of $p$, in particular, $\pi \not\in N_{{F'_0}^{[i]}F_v/F_v}(F_0'^{[i]} F_v)$.
On the other hand, for $w' \in \Delta_{K_0'} A^{[i]}$ not extending $v$, by the choice of $\pi$, $w'\vert_{F_0} \not\in S$, whence $g(X, \alpha_i) \in \mc{O}_{w'}[X]$ and $g(X, \alpha_i)$ is separable irreducible modulo $\mathfrak{m}_{w'}$, whereby in particular $g(0, \alpha_i) \in \mc{O}_{w'}^\times$.
Since the restriction of a valuation in $\Delta_K A^{[i]}_F$ to $F_0'$ is either trivial or equivalent to a valuation in $\Delta_{K_0'} A^{[i]}$, we in particular have $g(0, \alpha_i) \in \mc{O}_w^\times$ for all $w \in \Delta_K A^{[i]}_F$.

Let ${\mf g}$ be the genus of the function field $F/K$, let $\varphi_{\mf g}(x)$ be the existential $\Lar$-formula from \Cref{P:Koenigsmann-defining-constants}, fix some natural number 
$$M > \max_{1 \leq i \leq m} \max_{w \in \Delta_K A^{[i]}_F} w(\Disc(g(X, \alpha_i))$$
and consider the set
$$ D^{[i]} = \left\lbrace t \in F \enspace\middle|\enspace \exists a \in F : \begin{array}{l}
F \models \varphi_{\mf g}(a), a \neq 0, \text{ and } A^{[i]}_F \text{ is split over the} \\
\text{splitting field of } g(X, \alpha_i)_{at^M\Disc(g_i(X, \alpha_i))} \text{ over } F
\end{array} \right\rbrace. $$
This subset of $F$ is $\exists$-$\Lar(F_0')$-definable.
By \Cref{P:Koenigsmann-defining-constants} we have $K_1 \subseteq \varphi_{\mf g}(F) \subseteq K$.
By \Cref{P:mainProp} applied to the function fields $F/K$ and $F_1/K_1$ we see that the set $D^{[i]}$ satisfies
$$ \bigcap_{w \in \Delta_{K_1} A^{[i]}_{F_1}} \mathcal{O}_w \subseteq D^{[i]} \subseteq \bigcap_{w \in \Delta_K A^{[i]}_F} \mathcal{O}_w.$$

By the choice of $A^{[i]}$ we have for all $w \in \Delta_K A^{[i]}_F$ that $g(X, \alpha_i) \in \mc{O}_w[X]$ and is irreducible over $Fw$.
We compute that for $t \in F$ and $w \in \Delta_K A^{[i]}_F$ we have
$$ w\left(\frac{t^p}{g(t \pi, \alpha_i)}\right) = pw(t) - \min \{ 0, pw(\pi t) \} = p\max \{ w(t), -w(\pi) \}.$$
Since $w(\pi) > 0$ if and only if $w$ extends $v$ (by the choice of $\pi$) we infer that
\begin{displaymath}
\frac{t^p}{g(\pi t, \alpha_i)} \in \bigcap_{w \in \Delta_K A^{[i]}_F} \mathcal{O}_w \Leftrightarrow t \in \bigcap_{w \in C}\mathcal{O}_{w}.
\end{displaymath}
Now define the set
$$ D = \left\lbrace t \in F \enspace\middle|\enspace \text{for each } i \in \lbrace 1, \ldots, m \rbrace : \frac{t^p}{g(\pi t, \alpha_i)} \in D^{[i]} \right\rbrace. $$
This set is $\exists$-$\Lar(F_0)$-definable, since we can eliminate the parameters $\alpha_1, \ldots, \alpha_m$ by existentially quantifying over all the different roots of $h(X)$.
By the arguments presented above, we see that the set $D$ is as desired.
This concludes the proof in the special case under consideration.

\textbf{General case:}
With the help of \Cref{P:reduceToCyclic} we may find a finite Galois extension $\tilde{K_0}/K_0$ such that, setting 
$\tilde{K} = K\tilde{K_0}$ and $\tilde{F} = F\tilde{K_0}$ (composition within some common overfield of $F$ and $\tilde{K_0}$ over $K$), and fixing an extension $\tilde{v}$ of $v$ to $\tilde{F}$, we have that there exists a prime number $p$ such that

\begin{itemize}
\item $\tilde{F}\tilde{v}$ has a normal field extension of degree $p$, and this field extension can be given as the root field of a degree $p$ polynomial defined over $\tilde{F_0}'\tilde{v}$, where $\tilde{F_0}' = F_0 \tilde{K_0}'$ for some finite Galois extension $\tilde{K_0}'/K_0$ contained in $\tilde{K}$,
\item if $\car(K) \neq p$, then $\tilde{K_0}$ contains a $p$th root of unity,
\end{itemize}

\[\begin{tikzcd}[cramped]
	& {F_0\tilde{K_0}} & {\tilde{F_0}'} && {\tilde{F_1}} && {\tilde{F}} \\
	{F_0} &&& {F_1} && F \\
	& {\tilde{K_0}} & {\tilde{K_0}'} && {K_1\tilde{K_0}} && {\tilde{K}} \\
	{K_0} &&& {K_1} && K
	\arrow[from=1-2, to=1-3]
	\arrow[from=1-3, to=1-5]
	\arrow[from=1-5, to=1-7]
	\arrow[from=2-1, to=1-2]
	\arrow[from=2-1, to=2-4]
	\arrow[from=2-4, to=1-5]
	\arrow[from=2-4, to=2-6]
	\arrow[from=2-6, to=1-7]
	\arrow[dashed, from=3-2, to=1-2]
	\arrow[dashed, from=3-2, to=3-3]
	\arrow[dashed, from=3-3, to=1-3]
	\arrow[dashed, from=3-3, to=3-5]
	\arrow[dashed, from=3-5, to=1-5]
	\arrow[dashed, from=3-5, to=3-7]
	\arrow[from=3-7, to=1-7]
	\arrow[from=4-1, to=2-1]
	\arrow[dashed, from=4-1, to=3-2]
	\arrow[from=4-1, to=4-4]
	\arrow[from=4-4, to=2-4]
	\arrow[dashed, from=4-4, to=3-5]
	\arrow[from=4-4, to=4-6]
	\arrow[from=4-6, to=2-6]
	\arrow[from=4-6, to=3-7]
\end{tikzcd}\]

By the special case applied to the compositum $\tilde{F}$ (with  $K_1$ replaced by $K_1\tilde{K_0}$), and setting $\tilde{F_1} = F_1\tilde{K_0}$, we find the existence of an existentially definable subset $\tilde{D}$ of $\tilde{F}$ with parameters in $F_0$ such that $\bigcap_{\tilde{w} \in \tilde{C}} \mc{O}_{\tilde{w}} \cap \tilde{F_1} \subseteq \tilde{D} \subseteq \bigcap_{\tilde{w} \in \tilde{C}} \mc{O}_{\tilde{w}}$, where $\tilde{C}$ is the set of those $\tilde{w} \in \mc{V}(\tilde{F}/\tilde{K})$ with $\tilde{w}\vert_{\tilde{F_0}} \sim \tilde{v}\vert_{\tilde{F_0}}$.
The existential definability of $D = \tilde{D} \cap F$ now follows by a standard interpretation argument, see \cite[Lemma 2.1]{Daans-charp}.
\end{proof}
\begin{proof}[Proof of {\Cref{TI:definability}}]
This is \Cref{T:Edefinability} with $K_1 = K$ and $F_0 = F_1 = F$.
\end{proof}
\begin{opm}[{Uniformity in \Cref{T:Edefinability}}]\label{R:Uniformity}
In the proof of \Cref{T:Edefinability}, the reduction of the general case to the special case depends on interpreting a well-chosen finite field extension, which may be of arbitrarily high degree.
This makes the obtained formula non-uniform.
Another obstruction to uniformity is given via the reliance on \Cref{P:Koenigsmann-defining-constants}, which depends on a bound for the genus of the function field.

However, inspection of the proof reveals that the construction of the existential formula in the special case in the proof of \Cref{T:Edefinability} \textit{is} uniform, at least if one additionally assumes $K_0' = K_0$.
To be more precise, the proof shows the following:
\begin{quotation}
\noindent\textit{Let $p$ be a prime number and $g \in \nat$.
There exists $m \in \nat$ and an existential $\Lar$-formula $\varphi_{p,g}(x, y_1, \ldots, y_m)$ with the following property.}

\noindent\textit{Let $F/K$ be a function field in one variable of genus at most ${\mf g}$. Let $K_1$ be a large subfield of $K$, and $F_0$ a subfield of $F$ which is a function field in one variable over $F_0 \cap K_1$.
Consider $v \in \mc{V}(F/K)$ which has a uniformiser in $F_0$ and for which $F_0v$ admits a normal degree $p$ field extension not contained in $Fv$.
Then there exist $b_1, \ldots, b_m \in F_0$ such that $D = \lbrace a \in F \mid F \models \varphi_{p, {\mf g}}(a, b_1, \ldots, b_m) \rbrace$ satisfies
$$ \bigcap_{w \in C} \mc{O}_w \cap F_1 \subseteq D \subseteq \bigcap_{w \in C} \mc{O}_w$$
where $C$ is the set of $w \in \mc{V}(F/K)$ with $w\vert_{F_0} \sim v\vert_{F_0}$, and $F_1$ is a subfield of $F$ containing $F_0K_1$ such that $F_1/K_1$ is a function field in one variable.}
\end{quotation}
\end{opm}

\section{Application to undecidability of existential theories}\label{sect:undecidability}
We shall now discuss an application to undecidability of existential theories.

Let us state our set-up precisely.
Consider a field $F$ with a subfield $F_0$ which is finitely generated over its prime field.
The field $F_0$ is naturally \emph{computable}, i.e.~ there is a bijection (enumeration) $\nat \to F_0$ which makes the graphs of addition and multiplication on $F_0$ correspond to computable subsets of $\nat^3$.
This allows us to effectively enumerate the set of $\Lar(F_0)$-formulas.
A set of $\Lar(F_0)$-sentences is called \emph{decidable} if, via this enumeration, the set corresponds to a computable subset of $\nat$.
In particular, we can consider the set of all existential $\Lar(F_0)$-sentences which hold in $F$, called the \emph{existential $\Lar(F_0)$-theory of $F$}, and ask whether this is decidable or undecidable.
This property does not, in fact, depend on the specific choice of bijection $\nat \to F_0$ (since finitely generated fields are so-called \emph{computably stable}).
A reader who dislikes thinking about decidability of theories in infinite first-order languages may alternatively fix generators $\alpha_1, \ldots, \alpha_d$ of $F_0$ and ask about the decidability of the set of existential $\Lar(\{ \alpha_1, \ldots, \alpha_d \})$-sentences which hold in $F$; this is equivalent to the decidability of the existential $\Lar(F_0)$-theory of $F$.
Finally, we remark that the decidability of the existential $\Lar(F_0)$-theory of $F$ is equivalent to the ad hoc definition given in the introduction for decidability of the existential theory of $F$ with coefficients in $F_0$.
We refer the interested reader to \cite[Sections 2 \& 3]{Handbook-CompRingsFields} for background on computable rings and fields, and to \cite[Section 7]{HMLElementsOfRecursion} for more on decidability in first-order logic.

The following proposition states a general close relation between existential definability of valuation rings and undecidability of the existential theory of a function field in one variable.
In positive characteristic the underlying technique goes back to Pheidas \cite{PheidasHilbert10} (and later Videla in characteristic $2$ \cite{Videla}), who proved it for rational function fields over finite fields. The general statement was further developed in \cite[Lemma 1.5]{Shl96} and later proven in complete generality in \cite[Section 2]{EisShlap17}. 
In \cite[Corollary 5.2]{Daans-charp} a presentation is given with focus on keeping precise track of the used parameters.

For rational function fields of characteristic $0$, the proposition was essentially proven by Denef \cite[Section 3]{DenefDiophantine}.
The case of general function fields in one variable of characteristic $0$ was investigated in detail in \cite{Mor05} (and partially independently in \cite{EisHil10padic}).
Since the conditions of the geometric set-up of \cite{Mor05} take some effort to translate into our more algebraic framework (in particular it is not trivial to keep track of the used parameters), we include here the relevant parts of the argument for the reader's convenience in the case $\car(F) = 0$.

\begin{prop}\label{P:H10-reduction}
Let $F/K$ be a function field in one variable.
Let $F_0 \subseteq F$ be a finitely generated subfield such that $F = F_0K$.
Assume that there is a valuation $v \in \mc V(F/K)$ and an $\exists$-$\Lar(F_0)$-definable subset $\mc O$ of $F$ such that
$$ F_0 \neq \mc O_v \cap F_0 \subseteq \mc O \subseteq \mc O_v.$$
Then the existential $\Lar(F_0)$-theory of $F$ is undecidable.
\end{prop}
\begin{proof}
See \cite[Corollary 5.2]{Daans-charp} for a proof of the case where $\car(F) > 0$.
For the rest of the proof, assume $\car(F) = 0$.

We may assume without loss of generality that $K$ is relatively algebraically closed in $F$.
Setting $K_0 = F_0 \cap K$ we obtain that $K_0$ is relatively algebraically closed in $F_0$, whereby $F_0/K_0$ is a regular function field in one variable.
Let $C$ be a smooth projective curve over $K_0$ such that $F_0$ is the function field of $C$.
Consider the elliptic curve $E$ over $\qq$ given by the homogeneous equation
$$ E : Y^2Z = X^3 + XZ^2 + 2Z^3. $$
Fix $T \in F_0$ with the property that, considered as a morphism $C \to \mbb P^1_K$, $T$ has only simple ramification, simple zeroes, and simple poles, and furthermore $T \not\in \mc{O}_v$.
The existence of such $T$ is guaranteed by \cite[Proposition 2.3.1 and Remark 2.3.3]{Mor05} and the assumption that $F_0 \neq \mc{O}_v \cap F_0$.
Then consider the elliptic curve $\tilde{E}$ over $\qq(T)$ obtained as a quadratic twist of $E$ and with explicit equation
$$ \tilde{E} : (T^3 + T + 2)Y^2Z = X^3 + XZ^2 + 2Z^3.$$
Replacing $T$ by $\lambda T$ for some well-chosen $\lambda \in \zz \setminus \lbrace 0 \rbrace$ (which does not affect the already imposed conditions on $T$), we may further assume that $\tilde{E}(K(T)) = \tilde{E}(F)$, i.e.~every $F$-rational point of $\tilde{E}$ is defined over $K(T)$ \cite[Theorem 1.8(ii)]{Mor05}.
By \cite[Lemma 3.1]{DenefDiophantine} the group of rational points $\tilde{E}(K(T))$ is generated by a single point of infinite order $P = (T : 1 : 1)$, and $2$-torsion points.
Putting everything together, we obtain an isomorphism of abelian groups
$$ \iota : 2 \cdot \tilde{E}(F) \to \zz : 2mP \mapsto m$$
where $2mP$ denotes the element of $\tilde{E}(F)$ obtained by adding $P$ to itself $2m$ times according to the elliptic curve group law, and $2 \cdot \tilde{E}(F) = \lbrace 2Q \mid Q \in \tilde{E}(F) \rbrace$.

We claim that, by viewing $2 \cdot \tilde{E}(F)$ as a subset of $F^3$, $\iota$ gives a positive-existential interpretation of the ring $\zz$ inside the field $F$ with parameters in $F_0$.
Since the existential $\Lar$-theory of $\zz$ is undecidable by the theorem of Davis, Putnam, Robinson, and Matiyasevich \cite{Mat70}, the undecidability of the existential $\Lar(F_0)$-theory of $F$ would follow.
It remains to show that each of the following sets is $\exists$-$\Lar(F_0)$-definable:
\begin{align*}
& 2 \cdot \tilde{E}(F), \enspace \iota^{-1}(0), \enspace \iota^{-1}(1), \\
& \lbrace (x_1, y_1, z_1, x_2, y_2, z_2) \in (2 \cdot \tilde{E}(F))^2 \mid (x_1 : y_1 : z_1) = (x_2 : y_2 : z_2) \rbrace, \\
& \lbrace (P_1, P_2, P_3) \in (2 \cdot \tilde{E}(F))^3 \mid \iota(P_1) + \iota(P_2) = \iota(P_3) \rbrace, \\
& \lbrace (P_1, P_2, P_3) \in (2 \cdot \tilde{E}(F))^3 \mid \iota(P_1) \cdot \iota(P_2) = \iota(P_3) \rbrace.
\end{align*}
For all but the last set the $\exists$-$\Lar(F_0)$-definability follows easily from the explicit equation defining $\tilde{E}$ and the explicit formulas describing the elliptic curve group law in terms of the coordinates.

For $m \in \zz$, let $(x_m, y_m, z_m) \in F^3$ be such that $mP = (x_m : y_m : z_m)$.
By \cite[Lemma 3.2]{DenefDiophantine} we have that $\frac{x_m}{Ty_m} - m \in \mf{m}_v$.
It thus follows that if $P_i = (a_i : b_i : c_i) \in 2 \cdot \tilde{E}(F)$, then
$$ \iota(P_1) \cdot \iota(P_2) = \iota(P_3) \enspace\Leftrightarrow\enspace \frac{a_1 a_2}{2T^2 b_1 b_2} - \frac{a_3}{Tb_3} \in \mf{m}_v.$$
The right condition can indeed by described by an $\exists$-$\Lar(F_0)$-formula: since $(a_i : b_i : c_i) \in 2 \cdot \tilde{E}(F) = 2\tilde{E}(\qq(T))$ forces $\frac{a_i}{Tb_i} \in \qq(T) \subseteq F_0$, and $T^{-1}$ is a uniformiser for $v$, the hypotheses on $\mc{O}$ imply that
$$ \frac{a_1 a_2}{2T^2 b_1 b_2} - \frac{a_3}{Tb_3} \in \mf{m}_v \enspace\Leftrightarrow\enspace \frac{a_1 a_2}{23 b_1 b_2} - \frac{a_3}{b_3} \in \mc{O}.$$
From this it follows that via $\iota$, also the multiplication on $\zz$ is positive-existentially interpretable, which concludes the proof.
\end{proof}
\begin{proof}[Proof of \Cref{TI:decidability}]
This is a direct combination of \Cref{T:Edefinability} and \Cref{P:H10-reduction}.
\end{proof}
Let us mention the following application.
\begin{cor}\label{C:decidability-alg-extension}
Let $K$ be a field, $K_0 \subseteq K$ a finitely generated subfield such that $K$ does not contain the algebraic closure of $K_0$.
There exists an algebraic field extension $K'/K$ such that the existential $\Lar(K_0(T))$-theory of $K'(T)$ is undecidable.
If $K$ is large, one may take $K' = K$.
\end{cor}
\begin{proof}
First suppose that $K$ contains a real closure $K_1$ of $K_0$, we will show that in this case one can take $K'=K$.
We adapt the proof of \Cref{T:Edefinability} to this specific case.
Define $f(X) = X^2 + 1$ and consider over $K_1(T)$ the cyclic algebra $A = (K_1(T)[\sqrt{-1}]/K_1(T), \sigma, T)$ where $\sigma$ is the non-trivial $K_1(T)$-automorphism of $K_1(T)[\sqrt{-1}]$.
Since by assumption $K$ does not contain $\sqrt{-1}$ (for otherwise it would contain the algebraically closed field $K_1[\sqrt{-1}] \supseteq K_0$), we have that $v_T, v_\infty \in \Delta_{K} A_{K(T)}$, where $v_T, v_\infty$ denote the $T$-adic and degree valuation on $K(T)$, respectively.
Furthermore, $\Delta_{K_1} A$ consists only of the restrictions of $v_T$ and $v_\infty$.
Since real closed fields are large, \Cref{P:mainProp} (with $M = 1$) yields that
$$ \mc{O}_{v_T} \cap \mc{O}_{v_\infty} \cap K_1(T) \subseteq D_0 \subseteq \mc{O}_{v_T} \cap \mc{O}_{v_\infty} $$
where
$$ D_0 = \{ t \in K(T) \mid \exists a \in K(T)^\times : K(T) \models \varphi_{0}(a) \text{ and } A \text{ is split over } K(T)[\sqrt{4a^2t^2 - 1}] \rbrace,$$
$\varphi_0$ being the formula from \Cref{P:Koenigsmann-defining-constants}.
Setting now
$$ D = \{ t \in K(T) \mid \frac{t^2}{T^2t^2 + 1} \in D_0 \}$$
one computes that
$$ \mc{O}_{v_T} \cap K_1(T) \subseteq D \subseteq \mc{O}_{v_T} $$
and since $D$ is $\exists$-$\Lar(K_0(T))$-definable, we conclude via \Cref{P:H10-reduction} (with $F_0 = K_0(T)$).

Assume now that $K$ does not contain a real closure of $K_0$.
Let $p$ be a prime number such that $K_0$ has a finite extension of degree divisible by $p$ which is not contained in $K[\sqrt{-1}]$.
Let $K'$ be the $p$-closure of $K$, i.e.~the limit of all finite normal extensions of $K$ of degree not divisible by $p$.
Then $K'$ is $p$-closed, i.e.~all finite extensions of $K'$ have degree over $K'$ a power of $p$.
Recall from \cite[Theorem 1.3]{Pop_LittleSurvey} (attributed to Colliot-Thélène) that $p$-closed fields are large.
Since $K'[\sqrt{-1}]$ does not contain the algebraic closure of $K_0$ (by the choice of $p$ and \Cref{P:reduceToCyclic}) we conclude via \Cref{TI:decidability} (with $F_0 = K_0(T)$ and $K_1 = K'$).
\end{proof}

\printbibliography
\end{document}